\documentclass[11pt, fleqn]{amsart}

\usepackage[square,compress,comma, numbers,sort]{natbib}
\usepackage[colorlinks=true, citecolor=red, linkcolor=blue]{hyperref}
\usepackage{amsfonts,mathtools}
\usepackage{amsmath}

\allowdisplaybreaks[4]

\usepackage{amssymb}
\usepackage{color}
\usepackage{float}

\def\one{\mbox{1\hspace{-4.25pt}\fontsize{12}{14.4}\selectfont\textrm{1}}} 

\DeclarePairedDelimiterXPP\pk[1]{\mathbb{P}}\{ \}{}{ #1}
\DeclarePairedDelimiterXPP\E[1]{\mathbb{E}}\{ \}{}{	#1}

\usepackage{xparse}

\NewDocumentCommand{\ceil}{s O{} m}{%
	\IfBooleanTF{#1} 
	{$\left\lceil#3\right\rceil$} 
	{#2\lceil#3#2\rceil} 
}
\NewDocumentCommand{\floor}{s O{} m}{%
	\IfBooleanTF{#1} 
	{$\left\lfloor#3\right\rfloor$}
	{#2\lfloor#3#2\rfloor}
}

\definecolor{c20}{rgb}{0.,0.7,0.}
\definecolor{c30}{rgb}{0.,0.,1.}
\definecolor{c40}{rgb}{1,0.1,0.7}
\definecolor{c50}{rgb}{1,0,0}
\definecolor{c60}{rgb}{1,0.9,0.1}
\definecolor{c70}{rgb}{0.50,1.00,0.00}

\numberwithin{equation}{section}
\newtheorem{theo}{Theorem}[section]
\newtheorem{sat}[theo]{Proposition}
\newtheorem{de}[theo]{Definition}
\newtheorem{lem}[theo]{Lemma}

\newtheorem{example}[theo]{Example}
\newtheorem{korr}[theo]{Corollary}
\newtheorem{remark}[theo]{Remark}

\numberwithin{equation}{section}
\newtheorem{theorem}{Theorem}[section]

\newtheorem{lemma}{Lemma}[section]

\newcommand{\COM}[1]{}

\newcommand{\R}{\mathbb{R}}

\usepackage{geometry} 
\newgeometry{tmargin=3cm, bmargin=3cm, lmargin=2cm, rmargin=2cm}

\newcommand{\BQN}{\begin{eqnarray}}
\newcommand{\EQN}{\end{eqnarray}}
\newcommand{\BQNY}{\begin{eqnarray*}}
\newcommand{\EQNY}{\end{eqnarray*}}

\newcommand{\BS}{\begin{sat}}
	\newcommand{\ES}{\end{sat}}
\newcommand{\BT}{\begin{theo}}
	\newcommand{\ET}{\end{theo}}
\newcommand{\BK}{\begin{korr}}
	\newcommand{\EK}{\end{korr}}

\newcommand{\BEX}{\begin{example}}
	\newcommand{\EEX}{\end{example}}

\newcommand{\BD}{\begin{de}}
	\newcommand{\ED}{\end{de}}

\newcommand{\BIT}{\begin{itemize}}
	\newcommand{\EIT}{\end{itemize}}
\newcommand{\BDI}{\begin{description}}
	\newcommand{\EDI}{\end{description}}

\newcommand{\BRM}{\begin{remark}}
	\newcommand{\ERM}{\end{remark}}

\newcommand{\BEL}{\begin{lem}}
	\newcommand{\EEL}{\end{lem}}

\title{Running supremum of Brownian motion in dimension 2: exact and asymptotic results}

\author{Krzysztof K\c{e}pczy\'nski}
\address{Krzysztof K\c{e}pczy\'nski, Mathematical Institute, University of Wroc\l aw, pl. Grunwaldzki 2/4, 50-384 Wroc\l aw, Poland}
\email{Krzysztof.Kepczynski@math.uni.wroc.pl}

\begin{document}
\bigskip

\date{\today}
 \maketitle

 {\bf Abstract:} This paper investigates $\pi_T(a_1,a_2) = \mathbb{P}\left(\sup\limits_{t\in[0,T]} (\sigma_1B(t)-c_1t)>a_1, \sup\limits_{t\in[0,T]}( \sigma_2 B(t)-c_2t)>a_2\right),$  where $\{B(t) : t \geq 0\}$ is a standard Brownian motion, with $T >0, \sigma_1,\sigma_2>0, c_1, c_2\in\mathbb{R}.$ We derive explicit formula for the probability $\pi_T\left(a_1,a_2\right)$ and find its asymptotic behavior both in the so called many-source and high-threshold regimes.

\newcommand{\comment}[1]{}
\section{Introduction}
Consider $\{B(t)-ct:t\ge0\}$ a standard Brownian motion with drift $c\in \R$. The distribution of running supremum $\sup_{t\in[0,T]} (B(t)-ct)$, $T\in (0, \infty]$ plays the key role in many fields of applied probability, such as insurance and risk theory \cite{iglehart, konrad, michna2, Ji, foss, asmussen, dhm, djr}, queueing theory \cite{sm_buf_rev, kd2, mandjes, ref25}, financial mathematics \cite{Kunitomo, Li, merton, Wang, wilmott}. 
To be more precise, following the fundamental work of Iglehart \cite{iglehart},
\begin{eqnarray}
\pi_T(a):=\mathbb{P}\left( \sup_{t \in [0,T]}(\sigma B(t)-ct)>a \right)\label{1.dim}
\end{eqnarray}
describes the finite-time ruin probability in time horizon $[0,T]$, under the diffusion approximation regime, where $\sigma B(t)$ denotes the amount of cumulative claims up to time $t,$ $c>0$ is the premium rate and $a>0$ is the initial capital of the insurance company. In financial mathematics, $\max\left\{e^{\sup_{t\in[0,T]} (\sigma B(t)-ct)} - K, 0 \right\}$ describes the payoff of the lookback call option with maturity $T>0,$ where $e^{\sigma B(t) - ct}$ denotes the stock price at time $t$ and $K>0$ is the strike price, see e.g. \cite{wilmott}. In queueing theory,  $Q_T:=\sup_{t\in[0,T]} (\sigma B(t)-ct)$ has the interpretation of the buffer content at time $T>0$ of a fluid queue where $\sigma B(t)$ is the accumulated input in time interval $[0, t)$ and $c>0$ is the service rate, under the assumption that at time $t=0$ the system started off empty, see e.g. \cite{kd2}.

Notably, it is known that
\BQN\label{1dim_explicit}
\pi_T(a) & =
\Psi\left(\frac{a+cT}{\sqrt{T}\sigma} \right) + e^{-2ac/\sigma^2} \Psi\left(\frac{a-cT}{\sqrt{T}\sigma} \right),
\EQN
with $\Psi(x) = 1- \Phi(x) = \mathbb{P}\left(\mathcal{N}> x\right),$ where $\mathcal{N}$ is a standard normal random variable; see, e.g., \cite{kd2}.

In this contribution we consider two-dimensional extension of (\ref{1.dim}), that is, we are interested in
the joint distribution of
\begin{eqnarray}
\left(\sup_{t\in[0,T]} (\sigma_1B(t)-c_1t), \sup_{t\in[0,T]}( \sigma_2 B(t)-c_2t)\right),\label{2.dim}
\end{eqnarray}
with $\sigma_1,\sigma_2>0, c_1, c_2\in\mathbb{R}.$

The motivation to study (\ref{2.dim}) stems from its broad applicability in recently investigated problems, e.g.:\\
$\diamond$ \textit{Proportional reinsurance model.} 
Suppose that two companies, to be interpreted as the insurance company and the reinsurance company, share the payout of each claim in proportions $\sigma_1, \sigma_2>0,$ where $\sigma_1+\sigma_2=1,$ and receive premiums at rates $c_1,c_2>0,$ respectively. Let $R_i$ denote the risk process of $i$-th company
\BQNY
R_{i}(t) := a_i+c_it - \sigma_i B(t),
\EQNY
where $B(t)$ describes the accumulated claims up to time $t,$ $c_i>0$ is the premium rate and $a_i>0$ is the initial capital, $i=1,2.$ Then the joint survival function of (\ref{2.dim}) describes the \textit{component-wise ruin} probability in the proportional reinsurance, i.e.,
$$\mathbb{P}\left( \inf\limits_{t \in [0,T] } R_{1}(t) < 0, \inf\limits_{t \in [0,T] } R_{2}(t) < 0 \right) = \mathbb{P}\left( \sup_{t\in[0,T]} (\sigma_1B(t)-c_1t) > a_1, \sup_{t\in[0,T]}( \sigma_2 B(t)-c_2t) > a_2\right).$$
We refer for example to \cite{palmowski1, foss}, where related problems were investigated if the accumulated claim amount process is a L\'evy process and $T=\infty.$ For instance, Avram et al. \cite{palmowski1} found exact formulas for ruin probabilities in infinite-time horizon for spectrally negative or a compound Poisson process with exponential claims, and the asymptotic behavior of ruin probabilities under the Cram\'er condition. The case of subexponential claims was recently investigated in Foss et al. \cite{foss}.
\\
$\diamond$ \textit{Double barrier option.} Consider a binary option  where payoff depends on crossing two moving-barriers. Following, e.g., \cite{wilmott}, for $\nu_i>0, \mu_i>\mu, i=1,2,$ 
\BQNY
&\mathbb{E}\left(\one\{\exists t \in [0,T] : S(t) > S(0) e^{\nu_1 + \mu_1t}, \exists t \in [0,T] : S(t) > S(0) e^{\nu_2 + \mu_2t}\}\right)\\
&= \mathbb{P}\left(\sup\limits_{t \in [0,T]} \left(B(t) - \frac{\mu-\mu_1}{\sigma} t\right)> \frac{\nu_1}{\sigma}, \sup\limits_{t \in [0,T]} \left(B(t) - \frac{\mu-\mu_2}{\sigma} t\right) > \frac{\nu_2}{\sigma}\right)
\EQNY
models the price of the double-barrier binary option with double-exponential boundaries, where  $S(t) := S(0)e^{\sigma B(t) + \mu t}$ denotes the stock price at time $t,$ $\sigma, S(0)>0,$ $\mu \in \mathbb{R}.$ 
We refer to \cite{Kunitomo, Li, Wang} for related works on double barrier options.\\
$\diamond$ \textit{Parallel and tandem queues.}
Consider a two-node queuing network with constant service rates $c_1$ and $c_2,$ at first and second node, respectively. Traffic that enters the system is served in parallel by both queues. Following \cite{mandjes}, suppose that $B(t) - B(s)$ describes the accumulated input to the queue in time interval $[s,t).$ Let $Q_{i}(t)$ denote the buffer content of queue $i=1,2$ at time $t.$ If at time $t=0$ both queues are empty, then $$\left(Q_{1}(T), Q_2(T)\right) =_{d} \left( \sup_{t \in [0,T]} \left(B(t)-c_1t\right), \sup_{t \in [0,T]} \left(B(t)-c_2t\right) \right),$$ where $=_{d}$ denotes equality in distribution; see, e.g., \cite{kd2}. Then, the joint survival function of (\ref{2.dim}) describes the tail distribution function of the workload in a two-node parallel queue. Lieshout \& Mandjes \cite{mandjes} analyzed this model under stationarity assumption, i.e. the special case $T = \infty,$ and found exact distribution of (\ref{2.dim}), high-threshold asymptotics, and the most probable path leading to overflow. The tail distribution of (\ref{2.dim}) appears also in the context of Brownian-driven tandem queues; see, e.g., \cite{kd2, mandjes}.

We refer also to \cite{boxma, debicki.laplace, Ji2, Korshunov, michna2} for related works where multivariate analogs of (\ref{2.dim}) were analyzed. In particular, Boxma \& Kela \cite{boxma} and  D\c ebicki et al. \cite{debicki.laplace} studied suprema of the multivariate L\'evy processes in the language of Laplace transforms for $T = \infty,$ see also Michna \cite{michna2} where supremum of L\'evy processes with a broken drift was considered.

The main results of this contribution are given in Theorems \ref{tw1_1dim} and \ref{tw1}, where we present explicit formula for the joint tail distribution of (\ref{2.dim}). These results extend findings of Lieshout \& Mandjes \cite{mandjes}, see also Avram et al. \cite{palmowski1}, where the special case $T=\infty$ was derived. It occurs that depending on the model parameters we are faced with two scenarios: $a)$ dimension-reduction case when one component dominates the other and $b)$ the case when both components affect the joint distribution of (\ref{2.dim}). In the second case, the result is given in the language of a combination of survival functions of a bivariate normal random variable, which does not provide a straightforward interpretation of the behavior of Brownian motion that leads to high values of both components in (\ref{2.dim}). We analyze it in Sections \ref{high.threshold.scenario}, \ref{many.source.scenario} focusing on the so called \textit{high-threshold} regime (Theorem \ref{large.buffer}) and \textit{many-source} regime (Theorems \ref{tw21} and \ref{tw2}).

Let us briefly mention some notation used in this contribution. For two given positive functions $f(\cdot)$ and $g(\cdot)$ we write $f(x) =g(x)(1 +o(1))$ if $\lim\limits_{x \to \infty}f(x)/g(x) = 1$ and $f(x) =o(g(x))$ if $\lim\limits_{x \to \infty} f(x)/g(x) = 0.$\\
Moreover, for $\left(X,Y\right)\sim \mathcal{N}\left(\mu,\Sigma\right)$ with $\mu = \left(0,0\right)$ and $\Sigma=\left(\begin{array}{cc} 1 &  \rho \\ \rho  & 1 \end{array} \right),$ we write $$\Phi_2(\rho; s, t) = \mathbb{P}\left(X\leq s, Y \leq t\right)\text{ and } \Psi_2(\rho; s, t) = \mathbb{P}\left(X > s, Y > t\right).$$
By $\phi(\sigma^2;x)$ we denote the density function of the centered normal random variable with variance $\sigma^2.$ Recall that 
\begin{eqnarray}\label{l1}
\Psi\left(x\right) = \frac{1}{\sqrt{2 \pi} x} e^{-\frac{x^2}{2}}\left(1+o(1)\right), \text{ as } x \to \infty.
\end{eqnarray}

The remainder of the paper is organized as follows. In Section \ref{main_results} we formalize the considered model and present main results of this contribution. We derive an exact expression of the joint survival function of (\ref{2.dim}) and its asymptotic behavior in high-threshold and many-source regimes. Section \ref{proofs} contains auxiliary facts and proofs.
\section{Main results}\label{main_results}
Let $\{B(t): t \geq 0\}$ be a standard Brownian motion. We study the joint survival function 
$$\pi_T\left(a_1,a_2\right) := \mathbb{P}\left(\sup_{t\in[0,T]} (\sigma_1B(t)-c_1t)>a_1, \sup_{t\in[0,T]}( \sigma_2 B(t)-c_2t)>a_2\right).$$
Notice that, without loss of the generality, we shall suppose that $\sigma_i = 1,$ for $i = 1, 2.$ Next, due to the symmetry of the considered problem, in the rest of the paper we tacitly assume that $$c_1>c_2.$$
We note that if $a_1 \geq a_2 > 0,$ then the lines $a_1+c_1t$ and $a_2+c_2t$ do not intersect over $[0, \infty)$ and the problem degenerates to (\ref{1dim_explicit}). Therefore, we assume that $$0<a_1 < a_2.$$ Let $$t^* := \frac{a_2-a_1}{c_1-c_2}$$ denote the unique point of intersection of the above lines.
\subsection{Exact distribution}\label{exact}
We split the analysis of $\pi_{T}(a_1,a_2)$ into two scenarios: $t^{*} \geq T,$ when one component dominates the other, leading to the one-dimensional phenomena, and $t^{*}<T,$ when both components affect the probability $\pi_T(a_1,a_2).$ 
\begin{theorem}\label{tw1_1dim}
Suppose that $t^*\geq T.$ Then
$$\pi_{T}(a_1,a_2) = \Psi\left(\frac{a_2+c_2T}{\sqrt{T}}\right) + e^{-2a_2c_2}\Psi\left(\frac{a_2-c_2T}{\sqrt{T}}\right).$$
\end{theorem}
It immediately follows from the proof of Theorem \ref{tw1_1dim}, that for $t^{*} \geq T$ we have $$\pi_{T}(a_1,a_2) = \mathbb{P}\left( \sup\limits_{t \in [0,T]}\left(B(t)-c_2t\right) > a_2\right).$$
\begin{theorem} \label{tw1}
Suppose that $t^*<T.$ Then
\begin{eqnarray*}
    \pi_{T}(a_1,a_2) &=& \Psi \left(\frac{a_1+c_1T}{\sqrt{T}}\right) - \Psi_2 \left(-\sqrt{\frac{t^*}{T}}; \frac{a_1+c_1t^*}{\sqrt{t^*}}, -\frac{a_2+c_2T}{\sqrt{T}}\right) \\
        && + e^{-2a_1c_1} \left(\Psi\left(\frac{a_1-c_1T}{\sqrt{T}}\right) - \Psi_2\left(\sqrt{\frac{t^*}{T}}; \frac{a_1-c_1t^*}{\sqrt{t^*}}, \frac{(2a_1-a_2)-c_2T}{\sqrt{T}}\right) \right)\\
        && +  e^{-2a_2c_2}\Psi_2 \left(\sqrt{\frac{t^*}{T}}; \frac{a_2-c_2t^*}{\sqrt{t^*}}, \frac{a_2-c_2T}{\sqrt{T}}\right) \\
        &&+ e^{-2(a_1(c_1-2c_2)+a_2c_2)} \Psi_2\left(-\sqrt{\frac{t^*}{T}}; \frac{(2a_1-a_2)+c_2t^*}{\sqrt{t^*}}, \frac{(a_2-2a_1)-c_2T}{\sqrt{T}}\right).
\end{eqnarray*}
\end{theorem}
Theorem \ref{tw1} generalizes Theorem $3.1$ in Lieshout \& Mandjes \cite{mandjes},
where the special case $T = \infty$ was considered. Indeed, putting $T=\infty$ in Theorem \ref{tw1}, we recover that, for $c_1> c_2>0,$
\BQNY
       \pi_\infty(a_1,a_2) &:=&\mathbb{P} \left( \sup_{t \geq 0} \left(B(t) - c_1t\right) > a_1, \text{ } \sup_{t \geq 0} \left(B(t) - c_2t\right) > a_2\right)\\
        & =& e^{-2(a_1(c_1-2c_2)+a_2c_2)}\left(1- 
      \Psi\left(\frac{(c_1-2c_2)t^*-a_1}{\sqrt{t^*}}\right)\right) + e^{-2a_1c_1} \Psi\left(\frac{c_1t^*-a_1}{\sqrt{t^*}}\right) \\
      &&+ e^{-2a_2c_2}\Psi\left(\frac{(c_1-2c_2)t^*+a_1}{\sqrt{t^*}}\right)-\Psi\left(\frac{c_1t^*+a_1}{\sqrt{t^*}}\right),
\EQNY
which agrees with Theorem $3.1$ in \cite{mandjes}.
\subsection{Exact asymptotics}
In this section we study the asymptotic behavior of $\pi_T\left(a_1, a_2\right)$ in two cases: \textit{high-threshold} and \textit{many source} regimes. Having in mind motivations stemming from both fluid queueing theory and risk theory, we focus on the case of positive drifts, i.e. tacitly suppose that $c_1>c_2>0.$ 
\subsubsection{High-threshold scenario}\label{high.threshold.scenario}
We analyze asymptotic behavior of $$\pi_T(ab, b ) = \mathbb{P}\left( \sup_{t \in [0,T]} \left( B(t)-c_1t\right) > ab, \sup_{t \in [0,T]} \left( B(t)-c_2t\right) > b \right),$$
as $b \to \infty,$ where $a \in (0,1).$ We refer to \cite{konrad, palmowski1, foss, Ji, dhm, djr} for the motivation to study the asymptotics of $\pi_T(ab,b)$ and the results for related models considered in the risk theory.
\begin{theorem}\label{large.buffer}
Suppose that $c_1> c_2>0$ and  $a \in (0,1).$ Then, as $b \to \infty,$
$$\pi_T(ab,b) =\sqrt{\frac{2 T}{\pi}}\frac{1}{b} e^{-\frac{(b+c_2T)^2}{2T}}\left(1+o(1)\right).$$
\end{theorem}
It follows from Theorem $2.3$ that irrespective of the value of $T,$ in finite-time horizon  one coordinate asymptotically dominates the other. Indeed, as $b \to \infty,$ we obtain that
$$\pi_{T}(ab,b) = \mathbb{P}\left( \sup\limits_{t \in [0,T]}\left(B(t)-c_2t\right) > b\right)\left(1+o(1)\right).$$
\subsubsection{Many-source scenario}\label{many.source.scenario}
We analyze asymptotic behavior of  $$\psi_T(N) := \mathbb{P}\left( \sup_{t \in [0,T]} \left(\sum_{k=1}^{N} B_k(t)-c_1Nt\right) > a_1N, \sup_{t \in [0,T]} \left(\sum_{k=1}^{N} B_k(t)-c_2Nt \right)> a_2N \right),$$ as $N \to \infty,$ where $\{B_k(t): t\geq0\},$ $k=1,\dots,N$ are mutually independent standard Brownian motions.

We refer to \cite{sm_buf_rev, kd2, mandjes} for the motivation to consider such an asymptotics in the context of fluid queues in the so called many-source regime.

It appears from general theory of extremes of Gaussian processes that in the one-dimensional case the process $\left\{\sum_{k=1}^{N} B_k(t)-c_iNt : t \geq 0\right\}$ has the maximal chance to cross level $a_iN$ around the point that maximizes the following function  $\sigma_i^2(t) := \mathbb{V}ar\left(\frac{\sum_{k=1}^{N} B_k(t)}{a_i+c_it}\right),\text{ } t \geq 0.$ Elementary  calculations show that $$t_i := \arg\sup\limits_{t \geq 0}\sigma_i^2(t) =  \frac{a_i}{c_i} , \text{ for }i=1,2.$$
It appears that points $t_1$ and $t_2$ play important role  also in two-dimensional case.
As we show later, the order between $t_1, t_2$ and $t^{*}$ affects the asymptotics of $\pi(N,T),$ as $N \to \infty.$ It turns out that, similarly to Theorem \ref{tw1_1dim}, asymptotics of $\pi(N,T)$ also can lead to dimension-reduction. The full-dimensional case occurs only if $t_1 \leq t^{*} \leq t_2$ and $t^{*}<T.$ Then, the asymptotics depends on $$\tilde{t} := \frac{a_2-2a_1}{c_2}.$$

Similarly to Section \ref{exact} we distinguish two scenarios: $t^{*} \geq T$ and $t^{*}<T.$
\begin{theorem}\label{tw21} 
Suppose that $c_1>c_2>0$ and $t^* \geq T.$
    \begin{enumerate}
    \item[$(i)$] If $t_2> T,$ then, as $N \to \infty,$ 
    $$\psi_T(N) = \frac{1}{\sqrt{2\pi}}\left(\frac{\sqrt{T}}{a_2+c_2 T} +\frac{\sqrt{T}}{a_2-c_2 T} \right)\frac{1}{\sqrt{N}}e^{-\frac{(a_2+c_2T)^2}{2T}N}\left(1+o(1)\right).$$
    \item[$(ii)$] If $t_2= T,$ then, as $N \to \infty,$ 
    $$\psi_T(N) = \frac{1}{2}e^{-2a_2c_2N}\left(1+o(1)\right).$$
    \item[$(iii)$] If $t_2< T,$ then, as $N \to \infty,$ 
    $$\psi_T(N) = e^{-2a_2c_2N}\left(1+o(1)\right).$$
    \end{enumerate}
\end{theorem}
Scenario $t^{*}<T$ leads to several subcases which are analyzed in detail in the following theorem.
\begin{theorem}\label{tw2} 
Suppose that  $c_1 >c_2>0$ and $t^* < T.$
    \begin{enumerate}
        \item[$(i,a)$] If $t^*<t_1$ and $T<t_1,$ then, as $N \to \infty,$  
        $$\psi_T(N) = \frac{1}{\sqrt{2\pi}}\left(\frac{\sqrt{T}}{a_1+c_1 T}+\frac{\sqrt{T}}{a_1-c_1 T} \right) \frac{1}{\sqrt{N}}e^{-\frac{(a_1+c_1T)^2}{2T}N}\left(1+o(1)\right).$$
        \item[$(i,b)$] If $t^*<t_1$ and $T=t_1,$ then, as $N \to \infty,$  
        $$\psi_T(N) = \frac{1}{2}e^{-2a_1c_1N}\left(1+o(1)\right).$$
        \item[$(i,c)$] If $t^*<t_1$ and $T>t_1,$ then, as $N \to \infty,$  
        $$\psi_T(N) = e^{-2a_1c_1N}\left(1+o(1)\right).$$
        \item[$(ii)$] If $t_1=t^*,$ then, as $N \to \infty,$  
        $$\psi_T(N) = \frac{1}{2}e^{-2a_1c_1N}\left(1+o(1)\right).$$
        \item[$(iii,a)$] If $t_1<t^*<t_2$ and $T < \tilde{t},$ then, as $N \to \infty,$ 
        $$\psi_T(N) = \frac{1}{\sqrt{2\pi}}\left(\frac{\sqrt{T}}{(a_2-2a_1)-c_2T}-\frac{\sqrt{T}}{(2a_1-a_2)-c_2T}\right)\frac{1}{\sqrt{N}}e^{-\frac{\left(2a_1-a_2-c_2T\right)^2+4a_1c_1T}{2T}N}\left(1+o(1)\right).$$
        \item[$(iii,b)$] If $t_1<t^* <t_2$ and $T=\tilde{t},$ then, as $N \to \infty,$  
        $$\psi_T(N) = \frac{1}{2}e^{-2(a_1(c_1-2c_2)+a_2c_2)N}\left(1+o(1)\right).$$
        \item[$(iii,c)$] If $t_1<t^*<t_2,$ $T>\tilde{t}$ and $\tilde{t}<t^{*},$ then, as $N \to \infty,$
        $$\psi_T(N) = \frac{1}{\sqrt{2\pi}}\left(\frac{\sqrt{t^*}}{a_2-c_2t^*}+\frac{\sqrt{t^*}}{(2a_1-a_2)+c_2t^*} -\frac{\sqrt{t^*}}{a_1+c_1t^*}-\frac{\sqrt{t^*}}{a_1-c_1t^*}\right) \frac{1}{\sqrt{N}} e^{-\frac{\left(a_1+c_1t^*\right)^2}{2t^*}N}\left(1+o(1)\right).$$
        \item[$(iii,d)$] If $t_1<t^*<t_2,$ $T>\tilde{t}$ and $\tilde{t}=t^{*},$ then, as $N \to \infty,$  
        $$\psi_T(N) = \frac{1}{2}e^{-2(a_1(c_1-2c_2)+a_2c_2)N}\left(1+o(1)\right).$$
        \item[$(iii,e)$] If $t_1<t^*<t_2,$ $T>\tilde{t}$ and $\tilde{t}>t^*,$ then, as $N \to \infty,$ 
        $$\psi_T(N) = e^{-2(a_1(c_1-2c_2)+a_2c_2)N}\left(1+o(1)\right).$$
        \item[$(iv)$] If $t_1<t_2=t^*,$ then, as $N \to \infty,$  
        $$\psi_T(N) = \frac{1}{2}e^{-2a_2c_2N}\left(1+o(1)\right).$$
        \item[$(v)$] If $t_2<t^*,$ then, as $N \to \infty,$  
        $$\psi_T(N) = e^{-2a_2c_2N}\left(1+o(1)\right).$$
    \end{enumerate}
\end{theorem}
Theorem \ref{tw2} generalizes Theorems $3.3$ and $3.4$ in \cite{mandjes}, where the large-buffer asymptotics for  $T=\infty$ was considered. Indeed using self-similarity of Brownian motion, we obtain
\begin{eqnarray*}
    \psi_\infty(N) &:=&\mathbb{P} \left( \sup_{t \geq 0} \left(B(t) - c_1\sqrt{N}t\right) > a_1\sqrt{N}, \sup_{t \geq 0} \left(B(t) - c_2\sqrt{N}t\right) > a_2\sqrt{N}\right)\\
    &=& \mathbb{P}\left( \sup_{t \geq 0 } \left(B(t)-c_1t\right) > a_1N,\text{ }\sup_{t \geq 0 } \left(B(t)-c_2t\right) > a_2N\right),
\end{eqnarray*}
which was analyzed in \cite{mandjes} under the condition that $N \to \infty.$
\section{Proofs}\label{proofs}
In this section we give detailed proofs of theorems presented in Section \ref{main_results}.
\subsection{Proofs of Theorems \ref{tw1_1dim} and \ref{tw1}}
\begin{proof}[Proof of Theorem \ref{tw1_1dim}]
Note that $a_2+c_2t \geq a_1+c_1t,$ if $t^* \geq t.$ Thus
$$\pi_{T}(a_1,a_2) = \mathbb{P}\left( \sup_{t \in [0,T]} \left(B(t) - c_2 t\right) > a_2 \right)$$
and the proof follows by (\ref{1dim_explicit}).
\end{proof}
\begin{proof}[Proof of Theorem \ref{tw1}] 
Let $y_1 := a_1+c_1t^* = a_2+c_2t^*$ and $y_2 := a_2+c_2T.$\\ 
We have
\begin{equation}\label{representation}
\begin{split}
    \pi_{T}(a_1,a_2) &= 1- \mathbb{P}(\forall{t} \in [0,T]: B(t) \leq a_1+c_1t) - \mathbb{P}(\forall{t} \in [0,T]: B(t) \leq a_2+c_2t) \\
    &+ \mathbb{P}(\forall{t} \in [0,T]: B(t) \leq (a_1+c_1t)\wedge(a_2+c_2t))\\
    &= \mathbb{P}\left(\sup_{t \in [0,T]}\left(B(t) - c_1t\right) > a_1\right) + \mathbb{P}\left(\sup_{t \in [0,T]} \left(B(t) - c_2t\right) > a_2\right) -1 + \mathcal{P}(T),
\end{split}
\end{equation}
where
\begin{align*}
    \mathcal{P}(T) & = \mathbb{P}\left(\forall{t} \in [0,T]: B(t) \leq \left(a_1+c_1t\right) \wedge \left(a_2+c_2t\right) \right)\\
    &= \mathbb{P}\left(\forall{t} \in [0,t^*): B(t) \leq \left(a_1+c_1t\right), \forall{t} \in [t^*,T]: B(t) \leq \left(a_2+c_2t\right) \right).
\end{align*}
Conditioning with respect to $B(t^*)$ and using independence of increments of $\{B(t) : t \geq 0\}$ we get
\begin{align*}
    \mathcal{P}(T)&= \int_{-\infty}^{y_1}\phi\left(t^*; x_1\right) \mathbb{P}\left(\forall{t} \in [0, t^*): B(t) \leq a_1+c_1t| B(t^*) = x_1\right)\\
    & \qquad \times \mathbb{P}\left(\forall{t} \in [t^*,T]: B(t-t^*) \leq a_2-x_1+c_2t\right) dx_1.
\end{align*}
Next, substituting $t:= t-t^*$ and conditioning on $B(T-t^*),$ we obtain
\begin{align*}
    \mathcal{P}(T)&=\int_{-\infty}^{y_1}\int_{-\infty}^{y_2-x_1}  \phi\left(t^*; x_1\right)  \phi(T-t^*; x_2) \mathbb{P}\left(\forall{t} \in [0, t^*): B(t) \leq a_1+c_1t| B(t^*) = x_1\right) \\
    & \qquad \times \mathbb{P}\left(\forall{t} \in [0,T-t^*]: B(t) \leq a_2-x_1+c_2(t^*+t)|B(T-t^*) = x_2\right) dx_2 dx_1.
\end{align*}
Further,
\begin{align*}
    \mathcal{P}(T)& = \int_{-\infty}^{y_1} \int_{-\infty}^{y_2-x_1} \phi\left(t^*; x_1\right) \phi(T-t^*; x_2) \mathbb{P}\left(\sup\limits_{t\in[0,t^*)} \left(B(t)-\frac{t}{t^*}B(t^*) + C_1\frac{t}{t^*}\right) \leq a_1\right) \\
    &  \qquad  \times \mathbb{P}\left(\sup\limits_{t \in [0,T-t^*]} \left( B(t)-\frac{t}{T-t^*}B(T-t^*) + C_2\frac{t}{T-t^*}\right) \leq y_1-x_1\right) dx_2 dx_1,
\end{align*}
where $C_1 := x_1-c_1t^*$ and $C_2 :=x_2-c_2(T-t^*).$\\
Since $$\mathbb{P}\left( \sup\limits_{t\in[0,L]}\left(B(t) - \frac{t}{L}B(L) + y \frac{t}{L}\right) < b \right) = 1-e^{-\frac{2b\left(b-y\right)}{L}}, \text{ if } L>0 \text{ and }b-y \geq 0,$$
(see, e.g., \cite{Salminen}), we obtain
\begin{align*}
    \mathcal{P}(T)& =
 \int_{-\infty}^{y_1} \int_{-\infty}^{y_2-x_1} \phi\left(t^*; x_1\right) \phi(T-t^*; x_2) \left(1-e^{-2\frac{a_1(y_1-x_1)}{t^*}}\right)  \left(1-e^{-2\frac{(y_1-x_1)(y_2-x_1-x_2)}{T-t^*} }\right)dx_2 dx_1\\
    & = \int_{-\infty}^{y_1} \int_{-\infty}^{y_2-x_1} \phi\left(t^*; x_1\right) \phi(T-t^*; x_2) dx_2 dx_1\\
    & - \int_{-\infty}^{y_1} \int_{-\infty}^{y_2-x_1} \phi\left(t^*; x_1\right) \phi(T-t^*; x_2) e^{-2\frac{a_1(y_1-x_1)}{t^*}} dx_2 dx_1\\
    & - \int_{-\infty}^{y_1} \int_{-\infty}^{y_2-x_1} \phi\left(t^*; x_1\right) \phi(T-t^*; x_2) e^{-2\frac{(y_1-x_1)(y_2-x_1-x_2)}{T-t^*}}dx_2 dx_1\\
    & + \int_{-\infty}^{y_1} \int_{-\infty}^{y_2-x_1} \phi\left(t^*; x_1\right) \phi(T-t^*; x_2) e^{-2 \frac{a_1(y_1-x_1)}{t^*}}e^{-2\frac{(y_1-x_1)(y_2-x_1-x_2)}{T-t^*}}dx_2 dx_1\\
    & =: I_1(T) - I_2(T) - I_3(T) + I_4(T).
\end{align*}
Next we calculate $I_1(T),$ $I_2(T),$ $I_3(T)$ and $I_4(T).$\\
We have
\begin{align*}
    I_1(T) &:= \int_{-\infty}^{y_1} \int_{-\infty}^{y_2-x_1} \phi(t^{*}; x_1) \phi(T-t^{*}; x_2) dx_2 dx_1 \\ 
    &= \mathbb{P}\left(B(t^*) \leq y_1, (B(T)-B(t^*))+B(t^*) \leq y_2\right)
    = \Phi_2\left(\sqrt{\frac{t^*}{T}};\frac{a_1+c_1t^*}{\sqrt{t^*}}, \frac{a_2+c_2T}{\sqrt{T}}\right).
\end{align*}
Substitution $z_1 := x_1-2a_1,$ $z_2 := x_2$ leads straightforwardly to
\begin{align*}
    I_2(T) &:=  \int_{-\infty}^{y_1} \int_{-\infty}^{y_2-x_1} \phi(t^{*}; x_1) \phi(T-t^{*}; x_2) e^{-2\frac{a_1(y_1-x_1)}{t^{*}}} dx_2 dx_1 \\
    & = e^{-2a_1c_1} \Phi_2\left(\sqrt{\frac{t^*}{T}};\frac{-a_1+c_1t^*}{\sqrt{t^*}}, \frac{(a_2-2a_1)+c_2T}{\sqrt{T}}\right).
\end{align*}
Similarly, substitution $z_1 := x_1-2c_2t^*,$ $z_2 := 2x_1+x_2-2y_1$ and the fact that $a_1+c_1t^* = a_2+c_2t^*$ give us
\begin{align*}
     I_3(T) &:= \int_{-\infty}^{y_1} \int_{-\infty}^{y_2-x_1} \phi(t^*; x_1) \phi(T-t^*; x_2) e^{-2\frac{(y_1-x_1)(y_2-x_1-x_2)}{T-t^*}} dx_2 dx_1\\
    & = e^{-2a_2c_2} \Phi_2\left(-\sqrt{\frac{t^*}{T}};\frac{a_1+(c_1-2c_2)t^*}{\sqrt{t^*}}, \frac{-a_2+c_2T}{\sqrt{T}}\right).
\end{align*}
In a similar way we obtain
\begin{align*}
    I_4(T) &:= \int_{-\infty}^{y_1} \int_{-\infty}^{y_2-x_1} \phi(t^*; x_1) \phi(T-t^*; x_2)
    e^{-2\frac{(y_1-x_1)(y_2-x_1-x_2)}{T-t^*} } e^{-2\frac{a_1(y_1-x_1)}{t^*} } dx_2 dx_1\\
    & =  e^{-2(a_1(c_1-2c_2)+a_2c_2)} \Phi_2\left(-\sqrt{\frac{t^*}{T}};\frac{-a_1+(c_1-2c_2)t^*}{\sqrt{t^*}}, \frac{(2a_1-a_2)+c_2T}{\sqrt{T}}\right).
\end{align*}
Applying the above calculations and (\ref{1dim_explicit}) to (\ref{representation}) we obtain
\begin{align*}
    \pi_{T}(a_1,a_2) &= \Phi\left(-\frac{c_1T+a_1}{\sqrt{T}}\right) + e^{-2a_1c_1} \Phi\left(\frac{c_1T-a_1}{\sqrt{T}}\right) \\
    &+ \Phi\left(-\frac{c_2T+a_2}{\sqrt{T}}\right) + e^{-2a_2c_2} \Phi\left(\frac{c_2T-a_2}{\sqrt{T}}\right) \\
    & - \left(1-\Phi_2 \left(\sqrt{\frac{t^*}{T}}; \frac{c_1t^*+a_1}{\sqrt{t^*}}, \frac{c_2T+a_2}{\sqrt{T}}\right)\right) \\
    &- e^{-2a_1c_1} \Phi_2 \left(\sqrt{\frac{t^*}{T}}; \frac{c_1t^*-a_1}{\sqrt{t^*}}, \frac{c_2T+(a_2-2a_1)}{\sqrt{T}} \right)\\
    & -  e^{-2a_2c_2} \Phi_2 \left(-\sqrt{\frac{t^*}{T}}; \frac{(c_1-2c_2)t^*+a_1}{\sqrt{t^*}}, \frac{c_2T-a_2}{\sqrt{T}}\right) \\
    & + e^{-2(a_1(c_1-2c_2)+a_2c_2)}
        \Phi_2 \left(-\sqrt{\frac{t^*}{T}}; \frac{(c_1-2c_2)t^*-a_1}{\sqrt{t^*}} , \frac{c_2T+(2a_1-a_2)}{\sqrt{T}}\right).
\end{align*}
Finally, rewriting the above in a language of functions $\Psi$ and $\Psi_2$ completes the proof.
\end{proof}
\subsection{Proof of Theorem \ref{large.buffer}}
\begin{proof}[Proof of Theorem \ref{large.buffer}]
We observe that $t^{*} = \frac{1-a}{c_1-c_2}b \to \infty,$ as $b \to \infty.$ Then, without loss of generality, we shall assume $t^{*} > T.$ Hence $$\pi_T(ab, b) = \mathbb{P}\left(\sup\limits_{t \in [0,T]} (B(t)-c_2t)>b\right)\left(1+o(1)\right), \text{ as } b \to \infty.$$
The thesis follows from combination (\ref{1dim_explicit}) and (\ref{l1}). 
\end{proof}
\subsection{Proofs of Theorems \ref{tw21} and \ref{tw2}}
\begin{proof}[Proof of Theorem \ref{tw21}]
The proof follows by a straightforward combination of Theorem \ref{tw1_1dim}, (\ref{1dim_explicit}) and (\ref{l1}). 
\end{proof}
Before proceeding to the proof of the Theorem \ref{tw2}, we need the following lemma which gives the asymptotics of the joint survival function of a bivariate normal distribution. Its proof can be found in \cite{bivariate}.
\begin{lemma}{\text{ } }\label{l2}\label{l3}\label{l4}\label{l5}\label{l6}\label{l7}\label{l8}
\begin{itemize}
    \item[$(1)$] Let $\alpha > 0.$ 
    \begin{enumerate}
        \item[$(i)$] If $\rho \in [-1,1],$ then $\Psi_2\left(\rho; \alpha t, -\alpha t\right) = \frac{1}{\alpha t\sqrt{2\pi}} e^{-\frac{1}{2}\alpha^2 t^2}\left(1+o(1)\right), \text{ as } t \to \infty.$
    \end{enumerate}
    \item[$(2)$] Let $\alpha < 0,$ $\beta > 0,$ $|\alpha| > \beta.$ 
    \begin{enumerate}
        \item[$(i)$] If $\rho \in [-1,1],$ then
 $\Psi_2\left(\rho; \alpha t, \beta t\right) = \frac{1}{\beta t \sqrt{2 \pi}} e^{-\frac{1}{2}\beta^2 t^2}\left(1+o(1)\right), \text{ as } t \to \infty.$
    \end{enumerate}
    \item[$(3)$] Let $\alpha < 0,$ $\beta >0,$ $|\alpha| < \beta.$
    \begin{enumerate}
     \item[$(i)$] If $\rho > \frac{\alpha}{\beta},$ then $\Psi_2\left(\rho; \alpha t, \beta t\right) = \frac{1}{\beta t \sqrt{2 \pi}} e^{-\frac{1}{2}\beta^2 t^2}\left(1+o(1)\right),$ as $t \to \infty;$
     \item[$(ii)$] If $\rho = \frac{\alpha}{\beta},$ then $\Psi_2\left(\rho; \alpha t, \beta t\right) \leq \frac{1}{2} \Psi\left(\beta t \right)  = \frac{1}{2}\frac{1}{\beta t \sqrt{2 \pi}} e^{-\frac{1}{2}\beta^2 t^2}\left(1+o(1)\right),$ as $t \to \infty;$
     \item[$(iii)$] If $\rho < \frac{\alpha}{\beta},$ then $\Psi_2\left(\rho; \alpha t, \beta t\right) \leq \frac{\sqrt{1-\rho^2}}{2 \pi |\alpha-\rho\beta|\beta t^2} e^{-\frac{1}{2}\left[\frac{\alpha^2+\beta^2-2\rho\alpha\beta}{1-\rho^2}\right]t^2},$ as $t \to \infty.$
    \end{enumerate}
    \item[$(4)$]  Let $\alpha>0,$ $\beta>0,$ $\alpha \leq \beta.$
    \begin{enumerate}
    \item[$(i)$] If $\rho > \frac{\alpha}{\beta},$ then 
    $\Psi_2\left(\rho; \alpha t, \beta t\right) = \frac{1}{\beta t \sqrt{2 \pi} } e^{-\frac{1}{2}\beta^2t^2}\left(1+o(1)\right), \text{ as } t \to \infty.$
    \item[$(ii)$] If $\rho = \frac{\alpha}{\beta},$ then 
    $\frac{1}{2}\Psi\left(\beta t\right) \leq \Psi_2\left(\rho; \alpha t, \beta t\right) \leq \Psi\left(\beta t\right), \text{ for } t >0.$
    \item[$(iii)$] If $\rho < \frac{\alpha}{\beta},$ then 
    $\Psi_2\left(\rho; \alpha t, \beta t\right) = \frac{(1-\rho^2)^{\frac{3}{2}}\alpha^2 \beta^2}{2 \pi t^2 (\alpha - \rho \beta)(\beta - \rho \alpha)} e^{-\frac{1}{2}\left[\frac{\alpha^2+\beta^2-2\rho\alpha\beta}{1-\rho^2}\right]t^2}\left(1+o(1)\right), \text{ as } t \to \infty.$
    \end{enumerate}
 \item[$(5)$] Let $\beta > 0.$
    \begin{enumerate}
        \item[$(i)$] If $\rho>0,$ then $\Psi_2\left(\rho; 0, \beta t\right)  = \frac{1}{\beta t \sqrt{2 \pi}} e^{-\frac{1}{2}\beta^2 t^2}\left(1+o(1)\right),$ as $t \to \infty;$
        \item[$(ii)$] If $\rho =0,$ then $\Psi_2\left(\rho; 0, \beta t\right)  = \frac{1}{2} \frac{1}{\beta t \sqrt{2 \pi}} e^{-\frac{1}{2}\beta^2t^2}\left(1+o(1)\right),$ as $t \to \infty;$
        \item[$(iii)$] If $\rho <0,$ then $\Psi_2\left(\rho; 0, \beta t\right)  \leq \Psi\left(\beta t \right) \Phi\left(\frac{-\rho \beta t}{\sqrt{1-\rho^2}}\right),$ as $t \to \infty.$
    \end{enumerate}
\end{itemize}
\end{lemma}
\begin{proof}[Proof of Theorem \ref{tw2}] 
Using that $\sum\limits_{k=1}^{N} B_k(t) =_{d} \sqrt{N}B(t),$ we obtain 
\BQNY
\psi_T(N) &=& \mathbb{P}\left(\sup_{ t \in [0,T]} \left(\sqrt{N}B(t) - c_1Nt\right) > a_1N, \sup_{ t \in [0,T]} \left(\sqrt{N}B(t) - c_2Nt\right) > a_2N \right)\\
& = &\mathbb{P}\left(\sup_{ t \in [0,T]} \left(B(t) - c_1\sqrt{N}t\right) > a_1\sqrt{N}, \sup_{ t \in [0,T]} \left(B(t) - c_2\sqrt{N}t\right) > a_2\sqrt{N} \right).
\EQNY
We divide the proof on five cases.

\underline{\textit{Case $(i):$ $t^*<t_1.$}}
We have
$$\psi_T(N) \geq \mathbb{P}\left(\sup_{t \in [0, T]} \left(B(t) - c_1\sqrt{N} t\right) > a_1\sqrt{N}\right) - \mathbb{P}\left(\sup_{t \in [0,t^*]} \left(B(t) - c_1\sqrt{N} t\right) > a_1\sqrt{N}\right)$$
and 
$$\psi_T(N) \leq \mathbb{P}\left(\sup_{t \in [0, T]} \left(B(t) - c_1\sqrt{N} t\right) > a_1\sqrt{N}\right).$$
Formulas (\ref{1dim_explicit}) and  (\ref{l1}) imply, as $N \to \infty,$
$$\mathbb{P}\left(\sup_{t \in [0,t^*]} \left(B(t) - c_1\sqrt{N} t\right) > a_1\sqrt{N}\right) = \frac{1}{\sqrt{2\pi}}\left(\frac{\sqrt{t^*}}{a_1+c_1 t^*}+\frac{\sqrt{t^*}}{a_1-c_1 t^*} \right) \frac{1}{\sqrt{N}}e^{-\frac{(a_1+c_1t^*)^2}{2t^*}N}\left(1+o(1)\right).$$ 
Asymptotics of $\mathbb{P}\left(\sup\limits_{t \in [0,T]} \left(B(t) - c_1\sqrt{N} t\right) > a_1\sqrt{N}\right)$ depends on the relation between $t_1$ and $T,$ leading to three subcases.

\underline{\textit{Case $(i, a):$ $t^*<t_1$ and $T<t_1.$}}
By formulas (\ref{1dim_explicit}) and  (\ref{l1}), as $N \to \infty,$
$$\mathbb{P}\left(\sup_{t \in [0,T]} \left(B(t) - c_1\sqrt{N} t\right) > a_1\sqrt{N}\right) =  \frac{1}{\sqrt{2\pi}}\left(\frac{\sqrt{T}}{a_1+c_1 T}+\frac{\sqrt{T}}{a_1-c_1 T} \right) \frac{1}{\sqrt{N}}e^{-\frac{(a_1+c_1T)^2}{2T}N}\left(1+o(1)\right) .$$
Furthermore, since $\left(\frac{a_1+c_1 t^*}{\sqrt{t^*}}\right)^2 > \left(\frac{a_1+c_1T}{\sqrt{T}}\right)^2$ then as $N \to \infty$
$$\mathbb{P}\left(\sup_{t \in [0,t^*]} \left(B(t) - c_1\sqrt{N} t\right) > a_1\sqrt{N}\right) = o\left(\mathbb{P}\left(\sup_{t \in [0,T]} \left(B(t) - c_1\sqrt{N} t\right) > a_1\sqrt{N}\right)\right).$$ 
Thus, as $N \to \infty,$
$$\psi_T(N) =   \frac{1}{\sqrt{2\pi}}\left(\frac{\sqrt{T}}{a_1+c_1 T}+\frac{\sqrt{T}}{a_1-c_1 T} \right) \frac{1}{\sqrt{N}}e^{-\frac{(a_1+c_1T)^2}{2T}N}\left(1+o(1)\right).$$

\underline{\textit{Case $(i, b):$ $t^*<t_1$ and $T=t_1.$}}
Formulas (\ref{1dim_explicit}) and  (\ref{l1}) imply, as $N \to \infty,$
$$\mathbb{P}\left(\sup_{t \in [0,T]} \left(B(t) - c_1\sqrt{N} t\right) > a_1\sqrt{N}\right) =  \frac{1}{2}e^{-2a_1c_1N}\left(1+o(1)\right) .$$
Furthermore, since $\left(\frac{a_1+c_1 t^*}{\sqrt{t^*}}\right)^2/2 > \left(\frac{a_1+c_1T}{\sqrt{T}}\right)^2/2 = 2a_1c_1$ then as $N \to \infty$
$$\mathbb{P}\left(\sup_{t \in [0,t^*]} \left(B(t) - c_1\sqrt{N} t\right) > a_1\sqrt{N}\right) = o\left(\mathbb{P}\left(\sup_{t \in [0,T]} \left(B(t) - c_1\sqrt{N} t\right) > a_1\sqrt{N}\right)\right).$$ 
Thus, as $N \to \infty,$
$$\psi_T(N) = \frac{1}{2}e^{-2a_1c_1N}\left(1+o(1)\right).$$

\underline{\textit{Case $(i, c):$ $t^*<t_1$ and $T>t_1.$}}
Formulas (\ref{1dim_explicit}) and  (\ref{l1}) imply, as $N \to \infty,$
$$\mathbb{P}\left(\sup_{t \in [0,T]} \left(B(t) - c_1\sqrt{N} t\right) > a_1\sqrt{N}\right) = e^{-2a_1c_1N}\left(1+o(1)\right) .$$
Furthermore, since $\left(\frac{a_1+c_1 t^*}{\sqrt{t^*}}\right)^2/2 > 2a_1c_1$ then as $N \to \infty$
$$\mathbb{P}\left(\sup_{t \in [0,t^*]} \left(B(t) - c_1\sqrt{N} t\right) > a_1\sqrt{N}\right) = o\left(\mathbb{P}\left(\sup_{t \in [0,T]} \left(B(t) - c_1\sqrt{N} t\right) > a_1\sqrt{N}\right)\right).$$ 
Thus, as $N \to \infty,$
$$\psi_T(N) = e^{-2a_1c_1N}\left(1+o(1)\right).$$
\underline{\textit{Cases $(ii)$, $(iii)$ and $(iv):$ $t_1 \leq t^* \leq t_2.$}}
By Theorem \ref{tw1} we get
\BQNY
\psi_T(N)&=& \Psi \left(\frac{a_1+c_1T}{\sqrt{T}}\sqrt{N}\right) - \Psi_2 \left(-\sqrt{\frac{t^*}{T}}; \frac{a_1+c_1t^*}{\sqrt{t^*}}\sqrt{N}, -\frac{a_2+c_2T}{\sqrt{T}}\sqrt{N}\right) \\
        && + e^{-2a_1c_1N} \left(\Psi\left(\frac{a_1-c_1T}{\sqrt{T}}\sqrt{N}\right) - \Psi_2\left(\sqrt{\frac{t^*}{T}}; \frac{a_1-c_1t^*}{\sqrt{t^*}}\sqrt{N}, \frac{(2a_1-a_2)-c_2T}{\sqrt{T}}\sqrt{N}\right) \right)\\
        && +  e^{-2a_2c_2N}\Psi_2 \left(\sqrt{\frac{t^*}{T}}; \frac{a_2-c_2t^*}{\sqrt{t^*}}\sqrt{N}, \frac{a_2-c_2T}{\sqrt{T}}\sqrt{N}\right) \\
        &&+ e^{-2(a_1(c_1-2c_2)+a_2c_2)N} \Psi_2\left(-\sqrt{\frac{t^*}{T}}; \frac{(2a_1-a_2)+c_2t^*}{\sqrt{t^*}}\sqrt{N}, \frac{(a_2-2a_1)-c_2T}{\sqrt{T}}\sqrt{N}\right).
\EQNY

For the sake of brevity, write
\begin{flalign*}
 \alpha_0 &:= \frac{a_1+c_1T}{\sqrt{T}};& \alpha_1 &:= \frac{a_1+c_1t^*}{\sqrt{t^*}}; & \alpha_2 &:= \frac{a_1-c_1t^*}{\sqrt{t^*}}; & \alpha_3 & := \frac{a_2-c_2 t^*}{\sqrt{t^*}}; & \alpha_4 & := \frac{(2a_1-a_2)+c_2t^*}{\sqrt{t^*}};\\
\beta_0 &:= \frac{a_1-c_1T}{\sqrt{T}};& \beta_1 &:= \frac{a_2+c_2T}{\sqrt{T}}; & \beta_2 &:= \frac{(2a_1-a_2)-c_2T}{\sqrt{T}}; &\beta_3 & := \frac{a_2-c_2T}{\sqrt{T}}; & \beta_4 & := \frac{(a_2-2a_1)-c_2T}{\sqrt{T}}.
\end{flalign*}
Formula (\ref{l1}) and fact that $\alpha_0^2/2 = 2a_1c_1  + \beta_0^2/2$ imply, as $N \to \infty,$
$$\Psi \left(\alpha_0 \sqrt{N} \right) = \frac{1}{\sqrt{N}\sqrt{2\pi}} \frac{1}{\alpha_0} e^{-\frac{\alpha_0^2}{2}N}\left(1+o(1)\right)$$
and 
$$e^{-2a_1c_1N}\left(1-\Psi \left(\beta_0\sqrt{N}\right)\right)  =\frac{1}{\sqrt{N}\sqrt{2\pi}} \frac{1}{-\beta_0} e^{-\frac{\alpha_0^2}{2}N}\left(1+o(1)\right).$$
Lemma \ref{l2} $[1, 2, 3(i)]$ implies, as $N \to \infty,$
$$ \Psi_2 \left( -\sqrt{\frac{t^*}{T}}; \alpha_1\sqrt{N}, -\beta_1\sqrt{N}\right) = \frac{1}{\sqrt{N}\sqrt{2\pi}}\frac{1}{\alpha_1}e^{-\frac{\alpha_1^2}{2}N}\left(1+o(1)\right).$$ 
\underline{\textit{Case $(ii): t_1=t^*.$}} Lemma \ref{l2} $[5(iii)]$ implies, as $N \to \infty,$
$$e^{-2a_1c_1N}\Psi_2 \left( \sqrt{\frac{t^*}{T}}; \alpha_2 \sqrt{N}, \beta_2\sqrt{N} \right) = \frac{1}{2} e^{-2a_1c_1N}(1+o(1)).$$ 
Lemma \ref{l2} $[1, 2, 3(i), 4(i), 5(i)]$ and fact that $\alpha_1^2/2 = 2a_2c_2+\alpha_3^2/2$ imply, as $N \to \infty,$
$$e^{-2a_2c_2N}\Psi_2 \left(\sqrt{\frac{t^*}{T}}; \alpha_3\sqrt{N}, \beta_3\sqrt{N} \right)= \frac{1}{\sqrt{N}\sqrt{2\pi}}\frac{1}{\alpha_3}e^{-\frac{\alpha_1^2}{2}N}(1+o(1)).$$
Moreover, Lemma \ref{l2} $[1, 2, 3(i)]$ and fact that $\alpha_1^2/2 = 2(a_1(c_1-2c_2)+a_2c_2) + \alpha_4^2/2$ imply, as $N \to \infty,$ 
$$e^{-2(a_1(c_1-2c_2)+a_2c_2)N} \Psi_2 \left(-\sqrt{\frac{t^*}{T}}; \alpha_4\sqrt{N}, \beta_4\sqrt{N} \right)= \frac{1}{\sqrt{N}\sqrt{2\pi}}\frac{1}{\alpha_4}e^{-\frac{\alpha_1^2}{2}N}.$$
Furthermore $\alpha_0^2/2 > \alpha_1^2/2 = 2a_1c_1.$ Thus, as $N \to \infty,$
$$\psi_T(N) = \frac{1}{2}e^{-2a_1c_1N}(1+o(1)).$$
\underline{\textit{Case $(iii): t_1<t^*<t_2.$}} Formula (\ref{l1}), Lemma \ref{l2} $[1, 2, 3(i)]$ and fact that $\alpha_1^2/2 = 2a_1c_1 +\alpha_2^2/2$ imply, as $N \to \infty,$
$$e^{-2a_1c_1N}\left(1-\Psi_2 \left( \sqrt{\frac{t^*}{T}}; \alpha_2 \sqrt{N}, \beta_2\sqrt{N} \right) \right) = \left(\frac{1}{\sqrt{N}\sqrt{2\pi}}\frac{1}{-\beta_2}e^{-\frac{\beta_2^2+4a_1c_1}{2}N}+\frac{1}{\sqrt{N}\sqrt{2\pi}}\frac{1}{-\alpha_2}e^{-\frac{\alpha_1^2}{2}N}\right)(1+o(1)).$$ 
Lemma \ref{l2} $[1, 2, 3(i), 4(i), 5(i)]$  and fact that $\alpha_1^2/2 = 2a_2c_2+\alpha_3^2/2$ imply, as $N \to \infty,$
$$e^{-2a_2c_2N}\Psi_2 \left(\sqrt{\frac{t^*}{T}}; \alpha_3\sqrt{N}, \beta_3\sqrt{N} \right)= \frac{1}{\sqrt{N}\sqrt{2\pi}}\frac{1}{\alpha_3}e^{-\frac{\alpha_1^2}{2}N}(1+o(1)).$$
\noindent
Below we analyze five different scenarios depending on the relation between $\tilde{t},$ $t^*$ and $T.$\\
\underline{\textit{Case $(iii,a):$ $t_1<t^*<t_2$ and $T<\tilde{t}.$}} Lemma \ref{l2} $[2]$ and fact that $\beta_4^2/2 + 2(a_1(c_1-2c_2)+a_2c_2)= \beta_2^2/2+2a_1c_1$ imply, as $N \to \infty,$ 
$$e^{-2(a_1(c_1-2c_2)+a_2c_2)N} \Psi_2 \left(-\sqrt{\frac{t^*}{T}}; \alpha_4\sqrt{N}, \beta_4\sqrt{N} \right)=\frac{1}{\sqrt{N}\sqrt{2\pi}}\frac{1}{\beta_4}e^{-\frac{\beta_2^2+4a_1c_1}{2}N}\left(1+o\left(1\right)\right).$$
Furthermore $\alpha_0^2/2 > \alpha_1^2/2 > \beta_2^2/2+2a_1c_1.$ Thus, as $N \to \infty,$
$$\psi_T(N) = \frac{1}{\sqrt{2\pi}}\left(\frac{\sqrt{T}}{(a_2-2a_1)-c_2T}-\frac{\sqrt{T}}{(2a_1-a_2)-c_2T}\right)\frac{1}{\sqrt{N}}e^{-\frac{\left(2a_1-a_2-c_2T\right)^2+4a_1c_1T}{2T}N}\left(1+o(1)\right).$$
\underline{\textit{Case $(iii,b):$ $t_1<t^*<t_2$ and $\tilde{t}=T.$}}
Lemma \ref{l2} $[5(i)]$ and fact that $\alpha_1^2/2= \alpha_4^2/2 + 2(a_1(c_1-2c_2)+a_2c_2)$ imply, as $N \to \infty,$ 
$$e^{-2(a_1(c_1-2c_2)+a_2c_2)N} \left(\frac{1}{2} - \Psi_2 \left(-\sqrt{\frac{t^*}{T}}; \alpha_4\sqrt{N}, \beta_4\sqrt{N} \right) \right)= \frac{1}{\sqrt{N}\sqrt{2\pi}}\frac{1}{-\alpha_4}e^{-\frac{\alpha_1^2}{2}N}(1+o(1)).$$
Furthermore $\alpha_0^2/2 > \alpha_1^2/2 >2\left(a_1\left(c_1-2c_2\right) + a_2c_2\right).$ Thus, as $N \to \infty,$
$$\psi_T(N) = \frac{1}{2} e^{-2\left(a_1\left(c_1-2c_2\right) + a_2c_2\right)N}(1+o(1)).$$
\underline{\textit{Case $(iii,c):$ $t_1<t^*<t_2$ and $\tilde{t} < t^*.$}}
Lemma \ref{l2} $[1, 2, 3(i)]$ and fact that $\alpha_1^2/2 = \alpha_4^2/2 +  2(a_1(c_1-2c_2)+a_2c_2)$ imply, as $N \to \infty,$ 
$$e^{-2(a_1(c_1-2c_2)+a_2c_2)N} \Psi_2 \left(-\sqrt{\frac{t^*}{T}}; \alpha_4\sqrt{N}, \beta_4\sqrt{N} \right)= \frac{1}{\sqrt{N}\sqrt{2\pi}}\frac{1}{\alpha_4}e^{-\frac{\alpha_1^2}{2}N} \left(1+o\left(1\right)\right).$$
Furthermore $\beta_2^2/2+2a_1c_1 > \alpha_1^2/2$ and $\alpha_0^2/2> \alpha_1^2/2.$
Thus, as $N \to \infty,$
$$\psi_T(N) = \frac{1}{\sqrt{2\pi}} \left(-\frac{1}{\alpha_1}+\frac{1}{-\alpha_2}+\frac{1}{\alpha_3}-\frac{1}{-\alpha_4}\right) \frac{1}{\sqrt{N}}e^{-\frac{\alpha_1^2}{2}N}(1+o(1)).$$

\underline{\textit{Case $(iii, d):$ $t_1<t^*<t_2$ and $t^*=\tilde{t}.$}}  
 Lemma \ref{l2} $[5(i)]$ and fact that $\beta_4^2/2 + 2(a_1(c_1-2c_2)+a_2c_2)= \beta_2^2/2+2a_1c_1$ imply, as $N \to \infty,$ 
$$e^{-2(a_1(c_1-2c_2)+a_2c_2)N} \left(\frac{1}{2}- \Psi_2 \left(-\sqrt{\frac{t^*}{T}}; \alpha_4\sqrt{N}, \beta_4\sqrt{N} \right) \right)=\frac{1}{\sqrt{N}\sqrt{2\pi}}\frac{1}{-\beta_4}e^{-\frac{\beta_2^2+4a_1c_1}{2}N}\left(1+o\left(1\right)\right).$$
Furthermore $\alpha_0^2/2 > \alpha_1^2/2 = 2\left(a_1\left(c_1-2c_2\right) + a_2c_2\right)$ and $\beta_2^2/2+2a_1c_1> 2\left(a_1\left(c_1-2c_2\right) + a_2c_2\right).$ Thus, as $N \to \infty,$
$$\psi_T(N) = \frac{1}{2}e^{-2(a_1(c_1-2c_2)+a_2c_2)N}\left(1+o(1)\right).$$
\underline{\textit{Case $(iii, e):$ $t_1<t^*<t_2$ and $t^*<\tilde{t}<T.$}}
Formula (\ref{l1}), Lemma \ref{l2} $[1, 2, 3(i)],$ facts that $\alpha_1^2/2 = 2(a_1(c_1-2c_2)+a_2c_2) + \alpha_4^2/2$ and $\beta_4^2/2 + 2(a_1(c_1-2c_2)+a_2c_2)= \beta_2^2/2+2a_1c_1$ imply, as $N \to \infty,$ 
\begin{align*}
e^{-2(a_1(c_1-2c_2)+a_2c_2)N}&\left(1-\Psi_2 \left(-\sqrt{\frac{t^*}{T}}; \alpha_4\sqrt{N}, \beta_4\sqrt{N} \right)\right)\\
&= \left(\frac{1}{\sqrt{N}\sqrt{2\pi}}\frac{1}{-\alpha_4}e^{-\frac{\alpha_1^2}{2}N} - \frac{1}{\sqrt{N}\sqrt{2\pi}}\frac{1}{-\beta_4}e^{-\frac{\beta_2^2+4a_1c_1}{2}N}\right)(1+o(1)).
\end{align*}
Furthermore $\alpha_0^2/2 > \alpha_1^2/2 >  2\left(a_1\left(c_1-2c_2\right) + a_2c_2\right)$ and $\beta_2^2/2+2a_1c_1> 2\left(a_1\left(c_1-2c_2\right) + a_2c_2\right).$ Thus, as $N \to \infty,$
$$\psi_T(N) = e^{-2\left(a_1\left(c_1-2c_2\right) + a_2c_2\right)N}(1+o(1)).$$
\underline{\textit{Case $(iv):$ $t_1<t_2=t^*.$}} Formula (\ref{l1}), Lemma \ref{l2} $[1, 2, 3(i)]$ and fact that $\alpha_1^2/2 = 2a_1c_1 +\alpha_2^2/2$ imply, as $N \to \infty,$
$$e^{-2a_1c_1N}\left(1-\Psi_2 \left( \sqrt{\frac{t^*}{T}}; \alpha_2 \sqrt{N}, \beta_2\sqrt{N} \right) \right) = \left(\frac{1}{\sqrt{N}\sqrt{2\pi}}\frac{1}{-\beta_2}e^{-\frac{\beta_2^2+4a_1c_1}{2}N}+\frac{1}{\sqrt{N}\sqrt{2\pi}}\frac{1}{-\alpha_2}e^{-\frac{\alpha_1^2}{2}N}\right)(1+o(1)).$$ 
Lemma \ref{l2} $[5(iii)]$  and fact that $\alpha_1^2/2 = 2a_2c_2+\alpha_3^2/2$ imply, as $N \to \infty,$
$$e^{-2a_2c_2N}\Psi_2 \left(\sqrt{\frac{t^*}{T}}; \alpha_3\sqrt{N}, \beta_3\sqrt{N} \right)=\frac{1}{2}e^{-2a_2c_2N} (1+o(1)).$$
Lemma \ref{l2} $[1, 2, 3(i)]$ and fact that $\alpha_1^2/2 = 2(a_1(c_1-2c_2)+a_2c_2) + \alpha_4^2/2$ imply, as $N \to \infty,$ 
$$e^{-2(a_1(c_1-2c_2)+a_2c_2)N} \Psi_2 \left(-\sqrt{\frac{t^*}{T}}; \alpha_4\sqrt{N}, \beta_4\sqrt{N} \right)= \frac{1}{\sqrt{N}\sqrt{2\pi}}\frac{1}{\alpha_4}e^{-\frac{\alpha_1^2}{2}N} \left(1+o\left(1\right)\right).$$
Furthermore $\alpha_0^2/2> \alpha_1^2/2 = 2a_2c_2$ and $\beta_2^2/2+2a_1c_1 > \alpha_1^2/2.$ Thus, as $N \to \infty,$
$$\psi_T(N) = \frac{1}{2} e^{-2a_2c_2N}(1+o(1)).$$
\underline{\textit{Case $(v):$ $t_2<t^*.$}}
We have
$$\psi_T(N) \geq \mathbb{P}\left(\sup_{t \in [0,t^*]} \left(B(t) - c_2\sqrt{N} t\right) > a_2\sqrt{N}\right)$$
and 
$$\psi_T(N) \leq \mathbb{P}\left(\sup_{t \in [0, T]} \left(B(t) - c_2\sqrt{N} t\right) > a_2\sqrt{N}\right).$$
Formulas (\ref{1dim_explicit}) and  (\ref{l1}) imply, as $N \to \infty,$
$$\mathbb{P}\left(\sup_{t \in [0,t^*]} \left(B(t) - c_2\sqrt{N} t\right) > a_2\sqrt{N}\right) = e^{-2a_2c_2N}\left(1+o\left( 1\right)\right)$$ 
and 
$$\mathbb{P}\left(\sup_{t \in [0,T]} \left(B(t) - c_2\sqrt{N} t\right) > a_2\sqrt{N}\right) = e^{-2a_2c_2N} \left(1+o\left( 1\right)\right).$$
Thus, as $N \to \infty,$
$$\psi_T(N) =  e^{-2a_2c_2N}\left(1+o\left( 1\right)\right).$$
This completes the proof.
\end{proof}
\section*{Acknowledgments}
I would like to thank Krzysztof D\c{e}bicki for many stimulating discussions and helpful remarks during the preparation of this paper. The work was partially supported by NCN Grant No 2018/31/B/ST1/00370 (2019-2022).

\end{document}